\documentclass[preprint,11pt]{elsarticle}
\usepackage{amsfonts}
\usepackage{}
\usepackage{amsmath,amssymb,amsthm}
\usepackage{graphicx}
\usepackage{color,ulem}
\usepackage{soul} 
\usepackage{color, xcolor} 

\makeatletter
\def\ps@pprintTitle{%
  \let\@oddhead\@empty
  \let\@evenhead\@empty
  \def\@oddfoot{\reset@font\hfil\thepage\hfil}
  \let\@evenfoot\@oddfoot}
\makeatother

\newtheorem{defin}{Definition}[section]

\newtheorem{theorem}[defin]{Theorem}
\newtheorem{lemma}[defin]{Lemma}
\newtheorem{proposition}[defin]{Proposition}
\newtheorem{corollary}[defin]{Corollary}

\begin{document}

\begin{frontmatter}

\title{ Some characterizations of $\omega$-balanced topological groups with a $q$-point}

\author[mymainaddress]{Deng-Bin Chen}
\ead{374042079@qq.com}

\author[mymainaddress]{Hai-Hua Lin}
\ead{1048363420@qq.com}

\author[mymainaddress]{Li-Hong Xie \corref{cor1} }
\ead{yunli198282@126.com}

\cortext[cor1]{Corresponding author: Li-Hong Xie. Supported by NSF of Guangdong Province (No. 2021A1515010381) and the Innovation Project of Department of Education of Guangdong Province, China (No. 2022KTSCX145).}

\address[mymainaddress]{School of Mathematics and Computational Science, Wuyi University, Jiangmen, Guangdong 529000, P.R. China}

\begin{abstract}

In this paper, we study some characterizations of $q$-spaces, strict $q$-spaces and strong $q$-spaces under $\omega$-balanced topological groups as follows:

\begin{enumerate}
\item [(1)] A topological group $G$ is $\omega$-balanced and a $q$-space if and only if for each open neighborhood $O$ of the identity in $G$, there is a countably compact invariant subgroup $H$ which is of countable character in $G$, such that $H \subseteq O$ and the canonical quotient mapping $p:G\rightarrow G/H$ is quasi-perfect and the quotient group $G/H$ is metrizable.
\item [(2)] A topological group $G$ is $\omega$-balanced and a strict $q$-space if and only if for each open neighborhood $O$ of the identity in $G$, there is a closed sequentially compact invariant subgroup $H$ which is of countable character in $G$, such that $H \subseteq O$ and the canonical quotient mapping $p:G\rightarrow G/H$ is sequential-perfect and the quotient group $G/H$ is metrizable.
\item [(3)] A topological group $G$ is $\omega$-balanced and a strong $q$-space if and only if for each open neighborhood $O$ of the identity in $G$, there is a closed sequentially compact invariant subgroup $H$ of countable character $\{V_{n}:n\in \omega\} $, such that $H \subseteq O$  and $\{V_{n}:n\in\omega\}$ is a strong $q$-sequence at each $ y\in H $, in $G$ such that the canonical quotient mapping $p:G\rightarrow G/H$ is strongly sequential-perfect  and the quotient group $G/H$ is metrizable.
\end{enumerate}

\end{abstract}


\end{frontmatter}

\section{Introduction}
 A group $G$ endowed with a topology $\tau$ is called a {\it paratopological group} if multiplication in $G$ is continuous as a mapping of $G\times G$ to $G$, where $G\times G$ carries the usual product topology. A {\it topological group} is a paratopological group with continuous inversion. All topological groups and spaces in this paper are assumed to be Hausdorff. Given a topological group $G$ and its subgroup $H$, the quotient space $G/H$ means that the left coset with the quotient topology.

 Recall that a point $x$ in a space $X$ is called an {\it accumulation point} of a family $\gamma$ of subsets in $X$, if every neighborhood $U$ of $x$ joints with infinitely many elements of $\gamma$. A space $X$ is called {\it sequentially compact} if every sequence of $X$ has a convergent subsequence \cite{En}. Next we give the concepts of $q$-space, strict $q$-space and strong $q$-space.

\begin{enumerate}
\item [$\bullet$] A point $x\in X$ is called a {\it $q$-point} \cite[p.389]{AT} of a space $X$ if there exists a sequence $\{U_n:n\in\omega\}$ of open neighborhoods of $x$ in $X$ such that every sequence $\{x_n\}_{n\in\omega}$ of points in $X$ such that $x_n\in U_n$ for each $n\in\omega$ has a point of accumulation in $X$. We also call $\{U_n:n\in\omega\}$ a {\it $q$-sequence} at $x$. A space $X$ is said to be a {\it$q$-space} if every point of it is a $q$-point.

\item [$\bullet$]A point $x\in X$ is called a {\it strict $q$-point} of a space $X$ if there exists a $q$-sequence $\{U_n:n\in\omega\}$ at $x$ such that $\bigcap_{n\in\omega} U_n$ is a sequentially compact set. We also call $\{U_n:n\in\omega\}$ a {\it strict $q$-sequence} at $x$. A space $X$ is said to be a {\it strict $q$-space} if every point in $X$ is a strict $q$-point.

\item [$\bullet$]  A point $x\in X$ is called a {\it strong $q$-point} of a space $X$ if there exists a sequence $\{U_n:n\in\omega\}$ of open neighborhoods
of $x$ in $X$ such that every sequence $\{x_n\}_{n\in\omega}$ of points in $X$ such that $x_n\in U_n$ for each $n\in\omega$ has a convergent subsequence.
We also call $\{U_n:n\in\omega\}$ a strong {\it $q$-sequence} at $x$.
A space $X$ is said to be a {\it strong $q$-space} if every point in $X$ is a strong $q$-point \cite{PL}.

\end{enumerate}

 Recall that a continuous mapping $f:X\rightarrow Y$ is {\it perfect} if $f$ is a closed onto mapping and all fibers $f^{-1}(y)$ are compact subsets of $X$\cite[p.182]{En}.
 Recall that a continuous mapping $f:X\rightarrow Y$ is {\it quasi-perfect} \cite[Definition 2.1.3]{LY} if $f$ is closed such that $f^{-1}(y)$ is countably compact for each $y\in Y$.

 Peng and Liu gave a characterization of $q$-space.

\begin{theorem}\cite[Theorem 7]{PL}\label{Th1}
A topological group $G$ is a $q$-space if and only if there is a closed countably compact  subgroup $H$ of $G$ such that the quotient group $G/H$ is metrizable and the canonical quotient mapping $p:G\rightarrow G/H$ is quasi-perfect.
\end{theorem}

 Recall that a continuous mapping $f:X\rightarrow Y$ is called {\it sequential-perfect} if $f$ is closed and $f^{-1}(y)$ is a sequentially compact set for each $y\in Y$. Recall that a continuous mapping $f:X\rightarrow Y$ is {\it strongly sequential-perfect} if $f$ is closed and $f^{-1}(F)$ is a sequentially compact set for each sequentially compact set $F\subseteq Y$. The next two theorems appear in \cite{LXC}.

\begin{theorem}\cite[Theorem 2.4]{LXC}\label{Th2}
 A topological group $G$ is a strict $q$-space if and only if for each open neighborhood $U$ of the identity in $G$ there is a closed sequentially compact subgroup $H\subseteq U$ such that the quotient space $G/H$ is metrizable and the canonical quotient mapping $p:G\rightarrow G/H$ is sequential-perfect.
\end{theorem}

\begin{theorem}\cite[Theorem 2.6]{LXC}\label{Th3}
 A topological group $G$ is a strong $q$-space if and only if for each open neighborhood $U$ of the identity in $G$, there is a closed sequentially compact subgroup $H$ of countable character, which is a strong $q$-sequence at each $x\in H$, in $G$ such that the quotient space $G/H$ is metrizable and the canonical quotient mapping $p:G\rightarrow G/H$ is strongly sequential-perfect.
\end{theorem}

 Recall that a subset $A$ of a space $X$ is of {\it countable character} if there is a countable family $ \{U_{n}:n\in \omega\} $ of open neighborhoods of $ A $ such that for each open neighborhood $ U $ of $ A $ there is a $ U_{n} $ such that $ U_{n}\subseteq U $.

 Feathered topological groups are very important in topological algebra \cite{AT}. Recall that a topological group $G$ is {\it feathered} \cite[p.235]{AT} if $G$ has a non-empty compact set $K$ of countable character in $G$.

 Pasynkov gave characterizations of feathered topological group.

 \begin{theorem}\cite[Theorem 4.3.20]{AT}\label{Th4}
 A topological group G is feathered if and only if it contains a compact
 subgroup H such that the quotient space G/H is metrizable.
 \end{theorem}

 Recall that a paratopological group $ G $ is called {\it $\omega$-balanced} if for every neighbourhood $U$ of the identity $e$ in $G$, there exists a countable family $\gamma$ of open neighbourhoods of $e$ in $G$ such that for each $ x\in G $ one can find $ V\in \gamma $ satisfying $ xVx^{-1}\subseteq U $.

 In 2022, Xie and Yan proved the following result:
 \begin{theorem}\cite[Proposition 2.3]{XY}\label{Th5}
 A topological group $G$ is $\omega$-balanced and feathered if and only if for each open neighbourhood
 $O$ of the identity in $G$ there is a compact invariant subgroup $H$ such that $ H\subseteq O $ and the quotient group $G/H$ is metrizable.
 \end{theorem}

 In view of Theorem \ref{Th5}, it is natural to ask how to characterize $q$-space, strict $q$-space and strong $q$-space under $\omega$-balanced topological groups.

 This paper is organized as follows. First, it is proved that a topological group $G$ is $\omega$-balanced and a $q$-space if and only if for each open neighborhood $O$ of the identity in $G$, there is a countably compact invariant subgroup $H$ which is of countable character in $G$, such that $H \subseteq O$ and the canonical quotient mapping $p:G\rightarrow G/H$ is quasi-perfect and the quotient group $G/H$ is metrizable(see Theorem \ref{Th6}). As a corollary, we show that a topological group $G$ is $\omega$-narrow and a $q$-space if and only if for each open neighborhood $O$ of the identity in $G$, there is a countably compact invariant subgroup $H$ which is of countable character in $G$, such that $H\subseteq O$ and the canonical quotient mapping $p:G\rightarrow G/H$ is quasi-perfect and the quotient group $G/H$ is separable and metrizable(see Corollary \ref{C1}). Next, we prove that a topological group $G$ is $\omega$-balanced and a strict $q$-space if and only if for each open neighborhood $O$ of the identity in $G$, there is a closed sequentially compact invariant subgroup $H$ which is of countable character in $G$, such that $H \subseteq O$ and the canonical quotient mapping $p:G\rightarrow G/H$ is sequential-perfect and the quotient group $G/H$ is metrizable(see Theorem \ref{Th8}). As a corollary, we show that a topological group $G$ is $\omega$-narrow and a strict $q$-space  if and only if for each open neighborhood $O$ of the identity in $G$, there is a closed sequentially compact invariant subgroup $H$ which is of countable character in $G$, such that $H \subseteq O$ and the canonical quotient mapping $p:G\rightarrow G/H$ is sequential-perfect and the quotient group $G/H$ is separable and metrizable(see Corollary \ref{C2}). Finally, we show that a topological group $G$ is $\omega$-balanced and a strong $q$-space if and only if for each open neighborhood $O$ of the identity in $G$, there is a closed sequentially compact invariant subgroup $H$ of countable character $\{V_{n}:n\in \omega\} $, such that $H\subseteq O$  and $\{V_{n}:n\in\omega\}$ is a strong $q$-sequence at each $ y\in H $, in $G$ such that the canonical quotient mapping $p:G\rightarrow G/H$ is strongly sequential-perfect and the quotient group $G/H$ is metrizable(see Theorem \ref{Th9}). As a corollary, we show that a topological group $G$ is $\omega$-narrow and a strong $q$-space if and only if for each open neighborhood $O$ of the identity in $G$, there is a closed sequentially compact invariant subgroup $H$ of countable character $\{V_{n}:n\in \omega\} $,  such that $H \subseteq O$ and $\{V_{n}:n\in\omega\} $ is a strong $q$-sequence at each $ y\in H $, in $ G $ such that the canonical quotient mapping $p:G\rightarrow G/H$ is strongly sequential-perfect and the quotient group $G/H$ is separable and metrizable(see Corollary \ref{C3}).

\section{Characterizations of generalized countably compact topological groups}\label{Sec:2}

 In view of Theorems \ref {Th4} and \ref {Th5}, it is naturally to ask how to give "internal" characterizations of $q$-spaces, strict $q$-spaces and strong $q$-spaces under $\omega$-balanced topological groups. To study this question, we firstly establish some very useful lemmas.

 A subgroup $H$ of a topological group $G$ is called {\it neutral} if for every open neighborhood $U$ of the identity $e$ in $G$, there exists an open neighborhood $V$ of $e$ such that $HV \subseteq UH$ (equivalently, $V H \subseteq HU$) \cite{JP}.

 \begin{lemma}\cite[Theorem 3.1]{FST}\label{L1}
 If $H$ is a closed neutral subgroup of a topological group $G$ such that $G/H$ is first-countable, then $G/H$ is metrizable.
 \end{lemma}

\begin{lemma}\label{L2} \cite[Lemma 1.7.6]{LY}
Let $\{U_n:n\in \omega\}$ be a decreasing sequence of open subsets in a topological space $X$ such that $\bigcap_{n\in \omega}U_n=\bigcap_{n\in \omega}\overline{U_n}$. Put $C=\bigcap_{n\in \omega}U_n$. Then the following statements are equivalent:
\begin{enumerate}
\item[(1)] $\{U_n:n\in \omega\}$ is a $q$-sequence in $X$;
\item[(2)] the family $\{U_n:n\in \omega\}$ is a base at $C$ in $X$ and $C$ is countably compact in $X$.
\end{enumerate}
\end{lemma}

A topological group $ G $ is called {\it range-metrizable} if for every open neighbourhood $ U $ of the identity $ e $ in $ G $, there exists a continuous homomorphism $ p $ of $ G $ onto a metrizable group $ H $ such that $ p^{-1}(V)\subseteq U $, for some open neighbourhood $ V $ of the identity in $ H $.

Now we give an "internal" characterization of $q$-space in an $\omega$-balanced topological group.

\begin{theorem}\label{Th6}
		
 A topological group $G$ is $\omega$-balanced and a $q$-space if and only if for each open neighborhood $O$ of the identity in $G$, there is a countably compact invariant subgroup $H$ which is of countable character in $G$, such that $H \subseteq O$ and the canonical quotient mapping $p:G\rightarrow G/H$ is quasi-perfect and the quotient group $G/H$ is metrizable.

\end{theorem}

\begin{proof}
 Necessity. Suppose that for each open neighborhood $O$ of the identity in $G$, there is a countably compact invariant subgroup $H$ which is of countable character in $G$, such that $H \subseteq O$ and the canonical quotient mapping $p:G\rightarrow G/H$ is quasi-perfect and the quotient group $G/H$ is metrizable. Then, by Theorem \ref{Th1}, $G$ is a $q$-space. Put $V=(G/H)\setminus p(G\setminus O)$. Then one can easily show that $V$ is an open neighbourhood of the identity in $G/H$ such that
 \begin{flalign*}
 \begin{split}
 p^{-1}(V)&=p^{-1}(G/H)\setminus p^{-1}(p(G\setminus O))\\
 &\subseteq G\setminus (G\setminus O)\\
 &=O\\
 \end{split}
 \end{flalign*}
 This implies that $G$ is range-metrizable. Thus from the fact that a topological group is $\omega$-balanced if it is range-metrizable \cite[Theorem 3.4.22]{AT} it follows that $G$ is $\omega$-balanced.

 Sufficiency. (1) Let $e$ be the  identity of $G$. Since topological group $G$ is regular and a $q$-space, there is a $q$-sequence $\{U_n:n\in\omega\}$ at $e$  such that $\bigcap_{n\in\omega}U_n=\bigcap_{n\in \omega}\overline{U_n}$. Let $\{U_n:n\in\omega\}$ be a decreasing sequence. Put $C=\bigcap_{n\in \omega}U_n$. Then by Lemma \ref{L2}, $\{U_n:n\in \omega\}$ is a base at $C$ in $G$ and $C$ is countably compact. Take any open neighbourhood $O$ of $e$. Then one can find an open symmetric neighbourhood $W$ at $e$ such that $ W \subseteq U_{0} \cap O $. Let $ \gamma_0=\{W\} $. Suppose that for some $ n\in \omega\ $ we have defined families $\gamma_0,\gamma_1,\dots,\gamma_k$ of open symmetric neighbourhoods at $e$ satisfying the following conditions for each
 $1\leqslant k\leqslant n$:
 \begin{enumerate}
 	
 \item[(i)] $|\gamma_{k-1}|\leqslant \omega $ and $ W \subseteq U_{k-1} \cap O $ for each $W\in \gamma_{k-1} $;

 \item[(ii)] $\gamma_{k-1}\subseteq \gamma_k$;

 \item[(iii)] for every $ W\in\gamma_{k-1}$, there exists $ V\in\gamma_k$ such that $V^2\subseteq W$;

 \item[(iv)] for every $W\in\gamma_{k-1} $ and $ x\in G $, there exists $ V\in \gamma_k$ such that  $ xVx^{-1}\subseteq W $.

 \end{enumerate}

 Since $\gamma_n$ is countable and $G$ is a topological group, we can find a countable family $\lambda_{n,1}$ of open symmetric neighbourhoods at $e$ such that every element of $\gamma_n$ contains the square of some element $V\in \lambda_{n,1}$. In addition, the group $G$ is $\omega$-balanced, so there exists a countable family $\lambda_{n,2}$ of open symmetric neighbourhoods at $e$ such that for each $ W\in \gamma_n$ and each $ x\in G$ there is $ V\in \lambda_{n,2}$ satisfying  $xVx^{-1}\subseteq W$. We can also assume that each element $V\in \lambda_{n,1}\cup \lambda_{n,2}$ is contained in $U_{n+1}\cap O$. Let the family $\gamma_{n+1}=\gamma_n \cup \lambda_{n,1} \cup \lambda_{n,2}$. Then the families $\gamma_0,\gamma_1,\dots,\gamma_{n+1}$ of symmetric open neighbourhoods at $e$ satisfying (i)-(iv).

 It is easy to see that the family $\gamma=\bigcup_{n\in \omega}\gamma_n$ is countable. Put $H=\bigcap \gamma$. Then we claim that the following $(\ast)$ formula is valid:

 $(\ast)$ $H$ is a countably compact invariant subgroup of countable character in $G$ such that $H\subseteq O$.

 (a) Firstly, we shall show that $ H $ is an invariant subgroup. Since every element in $\gamma$ is symmetric, we have $ H^{-1}=H $. Let us verify that $ HH\subseteq H $. Take any points $ a,b\in H$. By (iii), then for each $ W\in \gamma $ there is $ V\in \gamma $ such that $ V^2\subseteq W$. Thus $ab\in   V^2\subseteq W$ for each $W\in \gamma$. Then $ ab\in {\bigcap_{W\in \gamma}}W=H $. This implies that $ HH\subseteq H $. Thus $ H $ is a subgroup. Next we show that $H$ is invariant. In fact, take any  $ x\in G $. By (iv), then for each $ W\in \gamma $ there is $ V\in \gamma $ such that $ xVx^{-1}\subseteq W $. This implies $ xHx^{-1}\subseteq xVx^{-1}\subseteq W $ for every $W\in \gamma$. Thus $ xHx^{-1}\subseteq {\bigcap_{W\in \gamma}}W =H $, which implies that $ H $ is invariant.

 (b) Secondly, we shall show that $H$ is countably compact. By (i), $H=\bigcap \gamma \subseteq {\bigcap_{n\in \omega} U_{n}}=C $ is obvious. Since C is countably compact, it is enough to show that $H$ is closed. Take any $ y\notin H=\bigcap \gamma $. Then there is $ V\in \gamma_{n_{0}} $ such that $ y\notin V $. By (iii), there is $ U\in \gamma_{n_{0}+1}$ such that $U^2\subseteq V$. Thus $ y\notin U^2 $. Since every element in $\gamma$ is symmetric and open, we have $ yU\cap U=\emptyset $. Thus $ yU\cap H=\emptyset $ because of $ H\subseteq U $. Since $ yU $ is an open neighbourhood of $ y $, $ H $ is closed.

 (c) Finally, to show that $H$ is of countable character in $G$ it is enough to show that $ \gamma $ is a base for $ G $ at $ H $. Let $ W $ be an open neighbourhood of $ H $ in $ G $. Since  $C=\bigcap_{n\in \omega}U_n=\bigcap_{n\in \omega}\overline{U_n}$, $ C $ is closed and countably compact. Put $ K=C\setminus W $. Then $ K $ is a countably compact subset and $ K\cap H =\emptyset $. Since every element in $ \gamma $ is open and symmetric, one can easily show that $ H=\bigcap_{V\in \gamma}\overline{V} $ by (iii). We claim that $ \overline{V}\cap K=\emptyset $ for some $V\in \gamma_{i} $. In fact, if $ \overline{V}\cap K\neq\emptyset$ for each  $V\in \gamma$, then by the countable compactness of $K$, the set $ K\cap(\bigcap_{V\in \gamma} \overline{V})=K\cap H\neq \emptyset$, which is a contradiction. By (iii), take $ U\in \gamma_{i+1} $ such that $ U^2\subseteq V$. It is clear that $ C\subseteq K\cup W \subseteq KU \cup W $. Since $ KU \cup W $ is an open neighbourhood of $C$ and sequence \{$ U_{n}:n\in \omega $\} is a base for $ G $ at $ C $, there is $U_{n_{0}}\in \{ U_{n}:n\in \omega \} $ such that $ C\subseteq  U_{n_{0}} \subseteq  KU \cup W $. Put $l=\text{max}\{i+1,n_{0}\} $. Then there is $U_{1}\in \gamma_{l+1} $ such that $ U_{1}\subseteq U $ and $ U_{1}\subseteq U_{n_{0}} $. Since $ U^2\cap K=\emptyset $, it implies that $ U\cap KU=\emptyset $. Thus $ H\subseteq U_{1} \subseteq U_{n_{0}} \subseteq W $ by $ U_{1}\cap KU=\emptyset $. We prove that $ \gamma  $ is a base for $ G $ at $ H $. By (i), $ H\subseteq O $ is obvious, so we complete the proof of $(\ast)$ formula.

 (2) Next we shall show that the canonical quotient mapping $p:G\rightarrow G/H$ is a quasi-perfect mapping.  Since $ H $ is countably compact, $p^{-1}(p(y))=yH$ is also a countably compact subset of $ G $ for every $y\in G$. Therefore, we only to show that $ p $ is a closed mapping. Take a closed set $F$ in $ G $. To show that $p(F)$ is closed in $G/H$ it is enough to show that $FH$ is closed in $G$, because $p$ is a quotient mapping and $FH=p^{-1}(p(F))$. Take any $x\notin FH$. Then $F^{-1}x$ is a closed set in $G$ such that $F^{-1}x\cap H=\emptyset$. This implies that $ H\subseteq G\setminus F^{-1}x $ and $ G\setminus F^{-1}x $ is an open set in $ G $. Since $\gamma$ is a base at $H$ in $G$, there is a $V\in \gamma$ such that $ H\subseteq V^2\subseteq G\setminus F^{-1}x $. Thus  $F^{-1}x\cap V^2=\emptyset$. Since every element in $ \gamma$ is open and symmetric, we have $xV\cap FV=\emptyset$. This implies that $xV\cap FH=\emptyset$, because $H\subseteq V$. Thus we find an open neighborhood $xV$ of $x$ which is disjoint with $FH$. This implies that $FH$ is closed in $G$.

 Finally, we shall show that $G/H$ is metrizable. To show that $G/H$ is metrizable it is only to show that $H$ is a neutral subgroup in $G$ and $G/H$ is first-countable by Lemma \ref{L1}. Since $H$ is of countable character in $G$ and $p$ is an open and closed mapping, one can easily show that $G/H$ is first-countable. In fact, the family $\{p(V):V\in \gamma\}$ is a base at $p(e)$. Take any open neighbourhood $ W $ of $p(e)$. Then $ H=p^{-1}(p(e))\subseteq p^{-1}(W) $. Since $W$ is open in $ G/H $ and $p$ is a continuous mapping, then $ p^{-1}(W) $ is open in $ G $. Since $\gamma$ is a base at $ H $ in $ G $, there is a $ V\in \gamma $ such that $ H\subseteq V\subseteq p^{-1}(W) $. This implies that $ p(V)\subseteq p(p^{-1}(W))=W $. Thus the family $\{p(V):V\in \gamma\}$ is a base at $p(e)$. From the fact that $G/H$ is a homogeneous space \cite[Theorem 1.5.1]{AT} it follows that $G/H$ is first-countable. Now we shall prove that $H$ is a neutral subgroup in $G$. Take an open neighborhood $U$ of the identity in $G$. Then $HU$ is an open neighborhood of $H$. Since $\gamma$ is a base at $H$ in $G$, there is $V\in \gamma$ such that $V^2\subseteq HU$. This implies that $VH\subseteq HU$. Thus we have proved that $H$ is a neutral subgroup.

\end{proof}

  We recall that a  topological group $ G $ is called {\it $\omega $-narrow}\cite{AT} if for every open neighbourhood $V$ of the neutral element in $ G $, there exists a countable subset $ A $ of $ G $ such that $ AV=G $.

 \begin{theorem}\cite[Theorem 3.4.23]{AT}\label{Th7}
 A topological group $G$ is topologically isomorphic to a subgroup of the topological product of some family of second-countable groups if and only if $ G $ is $ \omega $-narrow.
 \end{theorem}

 \begin{proposition}\cite[Proposition 3.4.2]{AT}\label{P1}
 If a topological group $H$ is a continuous homomorphic image of an $\omega$-narrow topological group $ G $, then $ H $ is also $\omega$-narrow.
 \end{proposition}

 \begin{proposition}\cite[Proposition 3.4.5]{AT}\label{P2}
 Every first-countable $\omega$-narrow topological group has a countable base.
 \end{proposition}

 By Theorem \ref{Th6}, we have the following result.

 \begin{corollary}\label{C1}
  A topological group $G$ is $\omega$-narrow and a $q$-space if and only if for each open neighborhood $O$ of the identity in $G$, there is a countably compact invariant subgroup $H$ which is of countable character in $G$, such that
  $H \subseteq O$ and the canonical quotient mapping $p:G\rightarrow G/H$ is quasi-perfect and the quotient group $G/H$ is separable and metrizable.

 \end{corollary}

\begin{proof}
 Necessity. Suppose that for each open neighborhood $O$ of the identity in $G$, there is a countably compact invariant subgroup $H$ which is of countable character in $G$, such that $H \subseteq O$ and the canonical quotient mapping $p:G\rightarrow G/H$ is quasi-perfect and the quotient group $G/H$ is separable and metrizable.  Then, by Theorem \ref{Th1}, $G$ is a $q$-space. Thus we only to show that topological group $G$ is $\omega$-narrow. For each open neighborhood $O$ of the identity in $G$, there is a countably compact invariant subgroup $H_{O}$, such that $H_{O} \subseteq O$ and the canonical quotient mapping $p_{O}:G\rightarrow G/H_{O}$ is quasi-perfect and every quotient group $G/H_{O}$ is separable and metrizable. Thus every topological group $G/H_{O}$ is second-countable. From the proof process of the necessity of Theorem \ref{Th6}, it can be concluded that for each open neighborhood $O$ of the identity in $G$, there is an open neighbourhood $V_{O}$ of the identity in $G/H_{O}$ such that $ p_{O}^{-1}(V_{O})\subseteq O $. Let $\mathcal{B}$ be the family of all open neighbourhoods of the neutral element in $G$. Then the diagonal product $h$ of the family  \{$p_{O}$:$O\in \mathcal{B}$ \} is a topological isomorphism of $G$ onto a topological subgroup of the topological product of the family \{$G/H_{O}$:$O\in \mathcal{B}$\}.
 Thus by Theorem \ref{Th7}, $G$ is $\omega $-narrow.

 Sufficiency. It is well known that every $\omega$-narrow topological group is $\omega$-balanced. Thus topological group $G$ is $\omega$-balanced and a $q$-space. Then, by Theorem \ref{Th6}, for each open neighborhood $O$ of the identity in $G$, there is a countably compact invariant subgroup $H$ which is of countable character in $G$, such that $H \subseteq O$ and the canonical quotient mapping $p:G\rightarrow G/H$ is quasi-perfect and the quotient group $G/H$ is metrizable. Thus we only to show that $G/H$ is separable. Since topological group $G/H$ is a continuous homomorphic image of an $\omega$-narrow topological group $G$, by Proposition \ref{P1}, $G/H$ is $\omega$-narrow. $G/H$ is first-countable because $G/H$ is metrizable. Then $G/H$ is a first-countable $\omega$-narrow topological group. Thus $G/H$ has a countable base by Proposition \ref{P2}. This implies that $G/H$ is second-countable. Since every second-countable topological group is separable, we have $G/H$ is separable.
\end{proof}

Now we give an "internal" characterization of strict $q$-space in an $\omega$-balanced topological group.

\begin{theorem}\label{Th8}
A topological group $G$ is $\omega$-balanced and a strict $q$-space if and only if for each open neighborhood $O$ of the identity in $G$, there is a closed sequentially compact invariant subgroup $H$ which is of countable character in $G$, such that $H \subseteq O$ and the canonical quotient mapping $p:G\rightarrow G/H$ is sequential-perfect and the quotient group $G/H$ is metrizable.

\end{theorem}

\begin{proof}
 Necessity. Suppose that for each open neighborhood $O$ of the identity in $G$, there is a closed sequentially compact invariant subgroup $H$ which is of countable character in $G$, such that $H \subseteq O$ and the canonical quotient mapping $p:G\rightarrow G/H$ is sequential-perfect and the quotient group $G/H$ is metrizable. Then by Theorem \ref{Th2}, $G$ is a strict $q$-space. Put $V=(G/H)\setminus p(G\setminus O)$. Then one can easily show that $V$ is an open neighbourhood of the identity in $G/H$ such that
\begin{flalign*}
\begin{split}
p^{-1}(V)&=p^{-1}(G/H)\setminus p^{-1}(p(G\setminus O))\\
&\subseteq G\setminus (G\setminus O)\\
&=O\\
\end{split}
\end{flalign*}
This implies that $G$ is range-metrizable. Thus from the fact that a topological group is $\omega$-balanced if it is range-metrizable \cite[Theorem 3.4.22]{AT} it follows that $G$ is $\omega$-balanced.

Sufficiency. (1) Let $e$ be the  identity of $G$. Since topological group $G$ is regular and a strict $q$-space, there is a strict $q$-sequence $\{U_n:n\in\omega\}$ at $ e $ such that $\bigcap_{n\in \omega}U_n=\bigcap_{n\in \omega}\overline{U_n}$. Let $\{U_n:n\in\omega\}$ be a decreasing sequence. Put $C=\bigcap_{n\in \omega}U_n$. Then by Lemma \ref{L2}, $\{U_n:n\in \omega\}$ is a base at $C$ in $G$ and $C$ is sequentially compact in $G$.  We continue to use the notation of Theorem \ref{Th6}. The following proof is completely similar to Theorem \ref{Th6}, and we can obtain that the inclusion of $H$ is a countably compact invariant subgroup of countable character in $G$ such that $H\subseteq O$ already holds. Next we only to prove that $H$ is a sequentially compact subgroup in $G$. Since $ H=\bigcap \gamma \subseteq {\bigcap_{n\in \omega} U_{n}}=C $ and $ C $ is sequentially compact in $G$ and $H$ is closed, then $H$ is a sequentially compact subgroup in $G$.

(2) Next we shall show that the canonical quotient mapping $p:G\rightarrow G/H$ is a sequential-perfect mapping. Since $ H $ is sequentially compact, $ p^{-1}(p(y))=yH $ is also a sequentially compact subset of $ G $ for every $ y\in G $. Therefore, we only to prove that $ p $ is a closed mapping. The proof of closed mapping is completely similar to Theorem \ref{Th6}, and it will not be repeated here!
Finally, we shall show that $G/H$ is metrizable. Following the proof of Theorem \ref{Th6}, one can easily show that G/H is metrizable.
\end{proof}

By Theorem \ref{Th8}, we have the following result.

\begin{corollary}\label{C2}
 A topological group $G$ is $\omega$-narrow and a strict $q$-space  if and only if for each open neighborhood $O$ of the identity in $G$, there is a closed sequentially compact invariant subgroup $H$ which is of countable character in $G$, such that $H \subseteq O$ and the canonical quotient mapping $p:G\rightarrow G/H$ is  sequential-perfect and the quotient group $G/H$ is separable and metrizable.
\end{corollary}

\begin{proof}
 Necessity. Suppose that for each open neighborhood $O$ of the identity in $G$, there is a closed sequentially compact invariant subgroup $H$ which is of countable character in $G$, such that $H \subseteq O$ and the canonical quotient mapping $p:G\rightarrow G/H$ is  sequential-perfect and the quotient group $G/H$ is separable and metrizable. Then, by Theorem \ref{Th2}, $G$ is a strict $q$-space. Thus we only to show that topological group $G$ is $\omega$-narrow. For each open neighborhood $O$ of the identity in $G$, there is a closed sequentially compact invariant subgroup $H_{O}$, such that $H_{O} \subseteq O$ and every canonical quotient mapping $p_{O}:G\rightarrow G/H_{O}$ is sequential-perfect and  every quotient group $G/H_{O}$ is separable and metrizable. Thus every topological group $G/H_{O}$ is second-countable. From the proof process of the necessity of Theorem \ref{Th6}, it can be concluded that for each open neighborhood $O$ of the identity in $G$, there is an open neighbourhood $V_{O}$ of the identity in $G/H_{O}$ such that $ p_{O}^{-1}(V_{O})\subseteq O $. Let $\mathcal{B}$ be the family of all open neighbourhoods of the neutral element in $G$. Then the diagonal product h of the family  \{$p_{O}$:$O\in \mathcal{B}$\} is a topological isomorphism of $G$ onto a topological subgroup of the topological product of the family \{$G/H_{O}$:$O\in \mathcal{B}$\}.
 Thus by Theorem \ref{Th7}, $G$ is $\omega $-narrow.
	
 Sufficiency: It is well known that every $\omega$-narrow topological group is $\omega$-balanced. Thus topological group $G$ is $\omega$-balanced and a strict $q$-space. Then, by Theorem \ref{Th8}, for each open neighborhood $O$ of the identity in $G$, there is a closed sequentially compact invariant subgroup $H$ which is of countable character in $G$, such that $H \subseteq O$ and the canonical quotient mapping $p:G\rightarrow G/H$ is sequential-perfect and the quotient group $G/H$ is metrizable. Thus we only to show that $G/H$ is separable. Since topological group $G/H$ is a continuous homomorphic image of an $\omega$-narrow topological group $G$, by Proposition \ref{P1}, $G/H$ is $\omega$-narrow. $G/H$ is first-countable because $G/H$ is metrizable. Then $G/H$ is a first-countable $\omega$-narrow topological group. Thus $G/H$ has a countable base by Proposition \ref{P2}. This implies that $G/H$ is second-countable. Since every second-countable topological group is separable, we have $G/H$ is separable.
\end{proof}

Now we give an "internal" characterization of strong $q$-space in an $\omega$-balanced topological group.

\begin{theorem}\label{Th9}
 A topological group $G$ is $\omega$-balanced and a strong $q$-space if and only if for each open neighborhood $O$ of the identity in $G$, there is a closed sequentially compact invariant subgroup $H$ of countable character $\{V_{n}:n\in \omega\} $, such that $H \subseteq O$  and $\{V_{n}:n\in\omega\}$ is a strong $q$-sequence at each $ y\in H $, in $G$ such that the canonical quotient mapping $p:G\rightarrow G/H$ is strongly sequential-perfect  and the quotient group $G/H$ is metrizable.

\end{theorem}

\begin{proof}	
 Necessity. Suppose that for each open neighborhood $O$ of the identity in $G$, there is a closed sequentially compact invariant subgroup $H$ of countable character $\{V_{n}:n\in \omega\} $, such that $H \subseteq O$ and $\{V_{n}:n\in\omega\}$ is a strong $q$-sequence at each $ y\in H $, in $G$ such that the canonical quotient mapping $p:G\rightarrow G/H$ is strongly sequential-perfect  and the quotient group $G/H$ is metrizable. Then by Theorem \ref{Th3}, $G$ is a strong $q$-space. Put $V=(G/H)\setminus p(G\setminus O)$. Then one can easily show that $V$ is an open neighbourhood of the identity in $G/H$ such that
\begin{flalign*}
\begin{split}
p^{-1}(V)&=p^{-1}(G/H)\setminus p^{-1}(p(G\setminus O))\\
&\subseteq G\setminus (G\setminus O)\\
&=O\\
\end{split}
\end{flalign*}
This implies that $G$ is range-metrizable. Thus from the fact that a topological group is $\omega$-balanced if it is range-metrizable \cite[Theorem 3.4.22]{AT} it follows that $G$ is $\omega$-balanced.

 Sufficiency. (1) Let $e$ be the  identity of $G$. Since topological group $G$ is regular and a strong $q$-space, there is a strong $q$-sequence $\{U_n:n\in\omega\}$ at $e$ such that $\bigcap_{n\in \omega}U_n=\bigcap_{n\in \omega}\overline{U_n}$. Let  $\{U_n:n\in\omega\}$ be a decreasing sequence. Put $C=\bigcap_{n\in \omega}U_n$. Then by Lemma \ref{L2}, $\{U_n:n\in \omega\}$ is a base at $C$ in $G$ and $C$ is sequentially compact in $G$.
 We continue to use the notation of Theorem \ref{Th6}. The following proof is completely similar to Theorem \ref{Th6} and Theorem \ref{Th8}, and we can obtain that the inclusion of $H$ is a closed sequentially compact invariant subgroup of countable character $\gamma=\{V_{n}:n\in\omega\} $ in $G$ such that $H\subseteq O$ already holds. Thus we only need to prove that $\gamma=\{V_{n}:n\in\omega\} $ is a strong $q$-sequence at each $ y\in H $. In fact, $ V_{n}\subseteq U_{n} $ for every $ n\in \omega $ and $ \{U_{n}:n\in \omega\} $ is a strong $q$-sequence at each $ x\in G $. This implies that $ \{V_{n}:n\in \omega\} $ is a strong $q$-sequence at each $ y\in H $.

(2) Next we shall show that $G/H$ is metrizable. Following the proof of Theorem \ref{Th6}, one can easily show that G/H is measurable. Finally, we shall show that the canonical quotient mapping $p:G\rightarrow G/H$ is a strongly sequential-perfect mapping. Following the proof of Theorem \ref{Th6}, one can easily show that $p$ is a closed mapping. Take arbitrary sequentially compact set $F$ in $G/H$. It is only to show that $p^{-1}(F)$ is a sequentially compact set in $ G $. Take a sequence $\{x_n\}_{n\in \omega}$ of $p^{-1}(F)$. Then there is a convergent subsequence of $\{p(x_n)\}_{n\in \omega}$ in $F$ by the sequential compactness of $F$. Without loss of generality, we assume that  $\{p(x_n)\}_{n\in \omega}$ converges to some point $p(x)$ in $F$, where $x\in G$. Let $\{V_n:n\in \omega\}$ be an open neighborhood base at $H$ in $G$ which is a strong $q$-sequence at each $y\in H$. Since $p$ is closed and $H=p^{-1}(p(e))$, we can find an open neighborhood $U_n$ of $p(e)$ such that $p^{-1}(U_n)\subseteq V_n$ for each $n\in \omega$. Clearly, the $\{p^{-1}(U_n):n\in \omega\}$ is also an open neighborhood base at $H$ in $G$ which is a strong $q$-sequence at each $x\in H$. Thus, without loss of generality, we can assume that $V_n=p^{-1}(p(V_n))$ holds for each $n\in \omega$. Therefore, one can easily show that $\{xV_n:n\in \omega\}$ is an open neighborhood base at $xH$ in $G$ which is a strong $q$-sequence at each $z\in xH$. Since $xV_n=xp^{-1}(p(V_n))=xV_nH$ holds for each $n\in \omega$, we have that $xV_n=p^{-1}(p(xV_n))$ holds for each $n\in \omega$. For each $n\in \omega$ we claim that the $xV_n$ contains infinite elements of the sequence $\{x_n\}_{n\in \omega}$. Indeed, if there is an $xV_{n_0}$ such that it contains finite elements of the sequence $\{x_n\}_{n\in \omega}$, then $p(xV_{n_0})$ also contains finite elements of the sequence $\{p(x_n)\}_{n\in \omega}$, because $xV_{n_0}=p^{-1}(p(xV_{n_0}))$. This is a contradiction with $\{p(x_n)\}_{n\in \omega}$ converging to $p(x)$, because $p(xV_{n_0})$ is an open neighborhood of $p(x)$ by $p$ is open. Thus we can find a subsequence $\{x_{n_i}\}_{i\in \omega}$ of $\{x_{n}\}_{n\in \omega}$ such that $x_{n_i}\in xV_i$ for each $i\in \omega$. Since $\{xV_i:i\in \omega\}$ is a strong $q$-sequence at each $z\in xH$ in $G$, $ \{x_{n_i}\}_{i\in \omega}$ has a convergent subsequence. Thus $\{x_n\}_{n\in \omega}$ has a convergent subsequence $\{x_{n_j}\}_{j\in \omega}$. Since $F$ is a sequentially compact set $F$ in the metrizable space $G/H$, $p^{-1}(F)$ is closed in $G$ by continuity of $p$ and compactness of $F$. This implies that the subsequence $\{x_{n_j}\}_{j\in \omega}$ is convergent in $p^{-1}(F)$.
\end{proof}

By Theorem \ref{Th9}, we have the following result.

\begin{corollary}\label{C3}
 A topological group $G$ is $\omega$-narrow and a strong $q$-space if and only if for each open neighborhood $O$ of the identity in $G$, there is a closed sequentially compact invariant subgroup $H$ of countable character $\{V_{n}:n\in \omega\} $,  such that $H \subseteq O$ and $\{V_{n}:n\in\omega\} $ is a strong $q$-sequence at each $ y\in H $, in $ G $ such that the canonical quotient mapping $p:G\rightarrow G/H$ is strongly sequential-perfect and the quotient group $G/H$ is separable and metrizable.
\end{corollary}

\begin{proof}
 Necessity. Suppose that for each open neighborhood $O$ of the identity in $G$, there is a closed sequentially compact invariant subgroup $H$ of countable character $\{V_{n}:n\in \omega\} $ in $G$, such that $H \subseteq O$ and $\{V_{n}:n\in\omega\} $ is a strong $q$-sequence at each $ y\in H $, in $ G $ such that the canonical quotient mapping $p:G\rightarrow G/H$ is strongly sequential-perfect and the quotient group $G/H$ is separable and metrizable.  Then, by Theorem \ref{Th3}, $G$ is a strong $q$-space. Thus we only to show that topological group $G$ is $\omega$-narrow. For each open neighborhood $O$ of the identity in $G$, there is a closed sequentially compact invariant subgroup $H_{O}$, such that $H_{O} \subseteq O$ and every canonical quotient mapping $p_{O}:G\rightarrow G/H_{O}$ is strongly sequential-perfect and the quotient group $G/H_{O}$ is separable and metrizable. Thus every topological group $G/H_{O}$ is second-countable. From the proof process of the necessity of Theorem \ref{Th6}, it can be concluded that for each open neighborhood $O$ of the identity in $G$, there is an open neighbourhood $V_{O}$ of the identity in $G/H_{O}$ such that $ p_{O}^{-1}(V_{O})\subseteq O $. Let $\mathcal{B}$ be the family of all open neighbourhoods of the neutral element in $G$. Then the diagonal product $h$ of the family \{$p_{O}$:$O\in \mathcal{B}$\} is a topological isomorphism of $G$ onto a topological subgroup of the topological product of the family \{$G/H_{O}$:$O\in \mathcal{B}$\}.
 Then by Theorem \ref{Th7}, $G$ is $\omega $-narrow.
	
 Sufficiency: It is well known that every $\omega$-narrow topological group is $\omega$-balanced. Thus topological group $G$ is $\omega$-balanced and a strong $q$-space. Then, by Theorem \ref{Th9}, for each open neighborhood $O$ of the identity in $G$, there is a closed sequentially compact invariant subgroup $H$ of countable character $\{V_{n}:n\in \omega\} $,  such that $H \subseteq O$ and $ \{V_{n}:n\in\omega\} $ is a strong $q$-sequence at each $ y\in H $, in $ G $ such that the canonical quotient mapping $p:G\rightarrow G/H$  is strongly sequential-perfect and the quotient group $G/H$ is metrizable. Thus we only to show that $G/H$ is separable. Since topological group $G/H$ is a continuous homomorphic image of an $\omega$-narrow topological group $G$, by Proposition \ref{P1},  $G/H$ is $\omega$-narrow. $G/H$ is first-countable because $G/H$ is metrizable. Then $G/H$ is a first-countable $\omega$-narrow topological group. Thus $G/H$ has a countable base by Proposition \ref{P2}. This implies that $G/H$ is second-countable. Since every second-countable topological group is separable, we have $G/H$ is separable. 	

\end{proof}


\end{document}